\documentclass[a4paper,11pt]{amsart}
\usepackage{amsmath,amsthm,amsfonts,amssymb,graphicx,mathtools,enumerate}
\usepackage[usenames,dvipsnames,svgnames,table]{xcolor}
\usepackage{thmtools}
\usepackage{xpatch}
\usepackage{apptools}
\usepackage[margin=32mm]{geometry}
\usepackage[pdftitle={A lower bound on the average degree forcing a minor},colorlinks=true,citecolor=black,linkcolor=black,urlcolor=blue!80!black]{hyperref}
\usepackage[noabbrev,capitalise]{cleveref}
\usepackage[numbers,sort&compress]{natbib}
\allowdisplaybreaks
\makeatletter
\def\NAT@spacechar{~}
\makeatother
\theoremstyle{plain}
\newtheorem{thm}{Theorem}
\newtheorem{lem}[thm]{Lemma}

\newtheorem{claim}{Claim}

\crefname{lem}{Lemma}{Lemmas}
\crefname{thm}{Theorem}{Theorems}
\crefname{cor}{Corollary}{Corollaries}
\crefname{prop}{Proposition}{Propositions}
\crefname{conj}{Conjecture}{Conjectures}
\crefname{claim}{Claim}{Claims}
\newcommand{\arXiv}[1]{arXiv:\,\href{http://arxiv.org/abs/#1}{#1}}
\newcommand{\msn}[1]{MR:\,\href{http://www.ams.org/mathscinet-getitem?mr=MR#1}{#1}}
\newcommand{\doi}[1]{\href{https://dx.doi.org/#1}{\tt https://dx.doi.org/#1}}
\DeclarePairedDelimiter\ceil\lceil\rceil
\DeclarePairedDelimiter\floor\lfloor\rfloor

\renewcommand{\ge}{\geqslant}
\renewcommand{\le}{\leqslant}
\renewcommand{\geq}{\geqslant}
\renewcommand{\leq}{\leqslant}

\newcommand{\eps}{\varepsilon}
\begin{document}
\title[{A lower bound on the average degree forcing a minor}]{A lower bound on the\\ average degree forcing a minor}
%%%%%
\author{Sergey Norin}
\address[S. Norin]{Department of Mathematics and Statistics, McGill University, Montr\'eal, Canada} \email{snorin@math.mcgill.ca}
\thanks{Norin is supported by NSERC grant 418520. Wood is supported by the Australian Research Council.}
%%%%%
\author{Bruce Reed}
\address[B. Reed]{School of Computer Science, McGill University, Montr\'eal, Canada}
\email{breed@cs.mcgill.ca}
%%%%%
\author{Andrew Thomason}
\address[A. Thomason]{Department of Pure Mathematics and Mathematical Statistics, University of Cambridge, United Kingdom}
\email{A.G.Thomason@dpmms.cam.ac.uk}
%%%%%
\author{David R. Wood}
\address[D. R. Wood]{School of Mathematics, Monash University, Melbourne, Australia}
\email{david.wood@monash.edu}
%\subjclass[2010]{Primary ??}
\date{\today}

\begin{abstract} 
We show that for sufficiently large $d$ and for $t\geq d+1$,  there is a graph $G$ with average degree $(1-\eps)\lambda  t \sqrt{\ln d}$ such that almost every graph $H$ with $t$ vertices and average degree $d$ is not a minor of $G$, where $\lambda=0.63817\dots$ is an explicitly defined constant. This generalises analogous results for complete graphs by Thomason (2001) and for general dense graphs by Myers and Thomason (2005). It also shows that an upper bound for sparse graphs by Reed and Wood (2016) is best possible up to a constant factor. 
\end{abstract}

\maketitle

%%%%%%%%%%%%
\section{Introduction}

\citet{Mader67} first proved that for every graph $H$, every graph with sufficiently large average degree contains $H$ as a minor\footnote{A graph $H$ is \emph{a minor} of a graph $G$ if a graph isomorphic to $H$ can be obtained from a subgraph of $G$ by contracting edges.}. The natural extremal question arises: what is the least average degree that forces $H$ as a minor? To formalise this question, let $f(H)$ be the infimum of all real numbers $d$ such that every graph with average degree at least $d$ contains $H$ as a minor. This value has been extensively studied for numerous graphs $H$, including small complete graphs \citep{Dirac64,Mader68,Jorgensen94,ST06,Song04}, the Petersen graph \citep{HW18a}, general complete graphs \citep{Mader68,Kostochka82,Kostochka84,delaVega,BCE80,Thomason84,Thomason01,Myers-CPC02}, complete bipartite graphs \citep{CRS11,Myers03,KP10,KO05,KP08,KP12},
general dense graphs \citep{MT-Comb05}, 
general sparse graphs \citep{ReedWood16,ReedWood16c,HarveyWood16}, %,NewPaperKevin}, 
disjoint unions of graphs \citep{Thom08,CLNWY17,KNQ19}, and
disjoint unions of cycles \citep{HarveyWood-Cycles}; see \citep{Thomason06} for a survey. 

For complete graphs $K_t$, the above question was asymptotically answered in the following theorem of  \citet{Thomason01}, where 
\begin{equation*}
\lambda := \max_{x >0} \frac{1-e^{-x}}{\sqrt{x}}=0.63817\ldots.
\end{equation*}

\begin{thm}[\citep{Thomason01}] 
\label{Thomason}
Every graph with average degree at least $(\lambda + o(1) )\, t\, \sqrt{\ln t}$ contains $K_t$ as a minor. Conversely, there is a graph with average degree at least $(\lambda + o(1) )\, t\, \sqrt{\ln t}$ that contains no $K_t$ minor. That is, 
\begin{equation*}
f(K_t) = (\lambda + o(1) )\, t\, \sqrt{\ln t}.
\end{equation*} 
\end{thm}

\citet{MT-Comb05} generalised this result for all families of dense graphs as follows\footnote{As is standard, we write that \emph{almost every} graph with $t$ vertices and average degree $d$ satisfies property $P$ if the probability that a random graph with $t$ vertices and average degree $d$ satisfies property $P$ tends to 1 as $t\to\infty$.}. 

\begin{thm}[\citep{MT-Comb05}]
\label{MyersThomason}
%For every $\eps\in(0,1)$ there exists $d_0$ such that for every $d\geq d_0$ and every $t$, every graph $G$ with average degree at least $\lambda t\sqrt{\ln d} + \eps t \sqrt{\ln t}$ contains every graph $H$ with $t$ vertices and average degree $d$ as a minor. That is,
%\begin{equation*}
%f(H) \leq 
%\lambda\, t\, \sqrt{\ln d} + \eps \, t\, \sqrt{\ln t}.
%\end{equation*}
%For every $\tau\in(0,1)$, for all integers $t$ and $d \geq t^\tau$, there is 
%a graph $G$ with average degree at least 
%$(\lambda+o(1))\, t\, \sqrt{\ln d}$ such that almost every graph $H$ with $t$ vertices and average degree $d$ is not a minor of $G$. That is, 
%\begin{equation*}
%f(H) \geq 
%(\lambda+o(1)) \, t\, \sqrt{\ln d} .
%\end{equation*}
For every $\tau\in(0,1)$, for all $t$ and $d\geq t^\tau$, for almost every graph $H$ with $t$ vertices and average degree $d$ (and for every $d$-regular graph with $t$ vertices), 
\begin{equation*}
f(H) = ( \lambda + o(1)) \, t\, \sqrt{\ln d}.
\end{equation*}
\end{thm}

\cref{MyersThomason} determines $f(H)$ for most dense graphs $H$ with $d\geq t^\tau$, but says nothing for sparse graphs $H$, where $d$ can be much smaller than $t$. In this regime, \citet{ReedWood16,ReedWood16c} established the following upper bound on  $f(H)$.

\begin{thm}[\citep{ReedWood16,ReedWood16c}]
\label{ReedWood}
For sufficiently large $d$, and for every graph $H$ with $t$ vertices and average degree $d$, every graph with average degree at least $3.895 \,t\,\sqrt{\ln d}$ contains $H$ as a minor.  That is, 
\begin{equation*}
f(H) \leq 3.895\, t\,\sqrt{\ln d}.
\end{equation*}
\end{thm}

The purpose of this paper is to show that this result is best possible up to a constant factor. Indeed, we precisely match the lower bounds in the work of \citet{Thomason01} and \citet{MT-Comb05}, strengthening the lower bound in \cref{MyersThomason} by eliminating the assumption that $d\geq t^\tau$. Informally, we prove that if $d$ is large, then almost every $H$ with $t \geq d+1$ vertices and average $d$ satisfies 
$f(H) \ge (1 -\eps)\lambda\, t\, \sqrt{\ln d}$. To state the result precisely, let $\mathcal{G}(t,m)$ be the space of random graphs with vertex-set $\{1,2,\dots,t\}$ and $m$ edges. Thus $\mathcal{G}(t,td/2)$ is the space of random graphs with $t$ vertices and average degree $d$.

%\begin{thm}
%\label{main} 
%OLD For every $\eps\in(0,1)$ there exists $d_0$ such that for every integer $d\geq d_0$ and for every integer $t\geq d+1$,  there is a graph $G$ with average degree at least $(1 -\eps)\lambda\, t\,\sqrt{\ln d}$ such that almost every graph $H$ with $t$ vertices and average degree $d$ is not a minor of $G$. That is, 
%\begin{equation*}
%f(H)\geq (1 -\eps)\lambda\, t\,\sqrt{\ln d}.
%\end{equation*}
%\end{thm}

\begin{thm}
	\label{main} 
	For every $\eps,c \in (0,1)$ there exists $d_0$  such that for every integer $d\geq d_0$ and for every integer $t\geq d+1$,  there is a graph $G$ with average degree at least $(1-\eps)\lambda\, t\,\sqrt{\ln d}$ such that if $H \in \mathcal{G}(t,td/2)$ then $\mathbb{P}(H \text{ is a minor of }G) < c^t$, and in particular, $\mathbb{P}(f(H) \geq (1-\eps)\lambda\, t\,\sqrt{\ln d}) > 1-c^t$.
\end{thm}

Note that in the proofs of \cref{Thomason,MyersThomason} the host graph  $G$ is a random graph of appropriately chosen constant density. Indeed, every such extremal graph is essentially a disjoint union of pseudo-random graphs \citep{Myers-CPC02,MT-Comb05}. However, random graphs themselves are not extremal when $d$  is small compared to $t$. Indeed, \citet{AF92} showed that if $d\le  \log_2 t$ then, for every  graph $H$ with $t$ vertices and maximum degree  $d$, a random graph on $t$ vertices (with edge probability $\frac12$) will  almost certainly contain a \emph{spanning} copy of $H$. To prove \cref{main},   we take $G$ to be a blowup of a suitably chosen small random  graph.   Note that \citet{Fox11} also considers minors of blowups of random graphs.   On the face of it, such blowups might   not to be pseudo-random, thus contradicting the fact that in many  cases the extremal graphs are known to be pseudo-random. But the notion  of pseudo-randomness involved is weak, asserting only that induced  subgraphs of constant proportion have roughly the same density, and the  blowups used here have this property.
  
Note that \citet{ReedWood16} claimed that a lower bound analogous to \cref{main} followed from the work of  \citet{MT-Comb05}.  However, this claim is invalid. The error occurs in the footnote on page~302 of \citep{ReedWood16}, where Theorem~4.8 and Corollary~4.9 of \citet{MT-Comb05} are applied. The assumptions in  these results mean that they are only applicable if the average degree of $H$ is at least $|V(H)|^\eps$ for some fixed $\eps>0$, which is not the case here. Also note that \citet{ReedWood16} claimed that a $ct\sqrt{\log d}$ lower bound holds for every $d$-regular graph (also as a corollary of the work of \citet{MT-Comb05}). This is false, for example, when $H$ is the $d$-dimensional hypercube \citep{HNW}.

%Note that the $ct\sqrt{\log d}$ lower bound in \cref{main} matches the best possible result for complete graphs $H$ (and all families of graphs $H$ with average degree at least $|V(H)|^\eps$), since the minimum average degree that forces  $H$ as  a minor is $\Theta(t\sqrt{\log t})$; see \citep{Kostochka82,Kostochka84,delaVega,BCE80,Thomason84,Thomason01,Myers-CPC02}.

%%%%%%%%%%%
\section{The Proof}

We will need the following Chernoff Bound.

\begin{lem}[\citep{MU05}]
\label{Chernoff}
Let $X_1,X_2,\dots,X_n$ be independent random variables, where each $X_i=1$ with probability $p$ and $X_i=0$ with probability $1-p$. Let $X := \sum_{i=1}^n X_i$. Then for $\delta\in(0,1)$, 
\begin{equation*}
\mathbb{P}( X \leq (1 - \delta)pn )  \leq \exp(-\tfrac{\delta^2}{2} pn).
\end{equation*}
\end{lem}

Let $G$ be a graph. For $\ell\in\mathbb{R}^+$, a non-empty set of at most $\ell$ vertices in $G$ is called an \emph{$\ell$-set}. Two sets $A$ and $B$ of vertices in $G$ are \emph{non-adjacent} if there is no edge in $G$ between $A$ and $B$. 

Our first lemma gives properties about a random graph. 

\begin{lem}
\label{NewRandomGraph}
Fix $p,\eps,\alpha\in(0,1)$ and $\beta\in(\alpha,1)$, and let $b:=(1-p)^{-1}$. Then there exists $d_0$ such that for every integer $d\geq d_0$, if $s:=\ceil{d^{\beta}}$ and $\ell := \sqrt{\alpha\log_b d}$, then there exists a graph $G$ with exactly $d$ vertices and more than $(\frac12 -\eps)pd^2$ edges, such that for every set $S$ of $s$ pairwise disjoint $\ell$-sets in $G$, more than $\frac12 d^{-\alpha} \binom{s}{2}$ pairs of $\ell$-sets in $S$ are non-adjacent.
\end{lem}

\begin{proof}
Let $G$ be a graph on $d$ vertices, where each edge is chosen independently at random with probability $p$. By \cref{Chernoff}, the probability that 
$ |E(G)| \leq  (\frac12-\eps)pd^2$ is less than $\frac12$.

If $A$ and $B$ are disjoint $\ell$-sets, then the probability that $A$ and $B$ are non-adjacent equals 
$(1-p)^{|A||B|} \geq (1-p)^{\ell^2} = d^{-\alpha}$. 
Consider a set $S$ of $s$ pairwise disjoint $\ell$-sets in $G$. 
Let $X_S$ be the number of pairs of elements of $S$ that are non-adjacent. 
Since the elements of $X_S$ are pairwise disjoint, \cref{Chernoff} is applicable and implies that 
the probability that $X_S \leq \frac12  d^{-\alpha} \binom{s}{2}$ is at most 
$\exp(- \frac{1}{8} d^{-\alpha} \binom{s}{2} ) \leq 
\exp(- \frac{1}{16} d^{\beta-\alpha} (s-1) )$, which is at most $\frac{1}{2} (2d^\ell)^{-s}$ since $d$ is sufficiently large. 

%\comment{p3 l22: Give more details how you obtain $1/2(2d^\ell)^{-s}$.}

The number of $\ell$-sets is $\sum_{i=1}^\ell\binom{d}{i} \leq 2d^\ell$. 
Thus the number of sets of $s$ pairwise disjoint $\ell$-sets is at most $\binom{2d^\ell}{s}\leq (2d^\ell)^s$. 
By the union bound, the probability that $X_S \leq \frac12 d^{-\alpha} \binom{s}{2}$, for some set $S$ of $s$ pairwise disjoint $\ell$-sets, is less than $\frac{1}{2}$. 

Hence with positive probability, $|E(G)|> (\frac12-\eps)pd^2$ edges, and $X_S > \frac12  d^{-\alpha} \binom{s}{2}$ for every set $S$ of $s$ pairwise disjoint $\ell$-sets. The result follows.
\end{proof}

The next lemma is the heart of our proof. 

%\begin{lem}
%\label{TheLemma} OLD
%Fix $p,\eps,\alpha\in(0,1)$ and let $b:=(1-p)^{-1}$. Then there exists $d_0$ such that for every integer $d\geq d_0$ and for every integer $t\geq d+1$,  there is a graph $G$ with average degree at least $(1 -\eps)p t \sqrt{\alpha \log_b d}$ such that almost every graph $H$ with $t$ vertices and average degree $d$ is not a minor of $G$.
%\end{lem}

\begin{lem}
	\label{TheLemma} 
	Fix $p,\eps,\alpha,c\in(0,1)$ and let $b:=(1-p)^{-1}$. Then there exists $d_0$ such that for every integer $d\geq d_0$ and for every integer $t\geq d+1$,  there is a graph $G$ with average degree at least $(1 -\eps)p t \sqrt{\alpha \log_b d}$ such that 
	such that if $H\in \mathcal{G}(t,td/2)$ then 
	$\mathbb{P}(H\text{ is a minor of }G) < c^t$.
\end{lem}

\begin{proof}
Let $\ell := \sqrt{\alpha\log_b d}$. Choose $\beta\in(\alpha,1)$ and let $s:=\ceil{d^{\beta}}$.
We assume that $d$ is sufficiently large as a function of $\alpha,\beta$ and $\eps$ to satisfy the inequalities occurring throughout the proof. 

%\comment{p3 l-2: $\epsilon/4$ or $\epsilon/2$? (applies to many places later as well)}

Let $G_0$ be the graph from \cref{NewRandomGraph} applied with $\frac{\eps}{4}$ in place of $\eps$.
Thus $|V(G_0)|=d$ and $|E(G_0)| >  (\frac12 \frac{\eps}{4})pd^2 = \frac12 (1-\frac{\eps}{2})pd^2$, and for every set $S$ of $s$ pairwise disjoint $\ell$-sets in $G_0$, more than $\frac12 d^{-\alpha} \binom{s}{2}$ pairs of $\ell$-sets in $S$ are non-adjacent. Call this property $(\star)$. 

Let $G$ be obtained from $G_0$ by replacing each vertex $x$ by an independent set $I_x$ of size 
\begin{equation*}
r:=\ceil*{\left(1 - \frac{\eps}{2}\right)\frac{t\ell}{d}}, 
\end{equation*} 
and replacing each edge $xy$ of $G_0$ by a complete bipartite graph between $I_x$ and $I_y$. Note that 
\begin{equation*}
|V(G)| = dr= d\left\lceil\left(1-\frac{\eps}{2}\right)\frac{t\ell}{d}\right\rceil < \left(1-\frac{\eps}{4}\right)\ell t,\end{equation*} 
and 
%\begin{equation*}
%|E(G)| = |E(G_0)|\, r^2 
%\geq   \left(1-\frac{\eps}{2}\right)\frac{pr^2d^2}{2}  
%= \left(1-\frac{\eps}{2}\right)\frac{prd}{2}|V(G)| 
%\geq (1 - \eps + \eps^2 )\frac{pt\ell }{2}|V(G)| .
%\end{equation*}
%Hence $G$ has average degree 
%$2\frac{|E(G)|}{|V(G)|} \geq (1 -\eps+\eps^2)pt\ell \geq 
%(1 -\eps)pt \sqrt{\alpha \log_b d}$, as claimed. 
\begin{equation*}
|E(G)| = |E(G_0)|\, r^2 
>   \left(1-\frac{\eps}{2}\right)\frac{pr^2d^2}{2}  
= \left(1-\frac{\eps}{2}\right)\frac{prd}{2}|V(G)| 
\geq \left(1-\frac{\eps}{2}\right)^2\frac{p t\ell }{2}|V(G)| .
\end{equation*}
Hence $G$ has average degree 
$2\frac{|E(G)|}{|V(G)|} \geq (1 -\frac{\eps}{2})^2pt\ell \geq 
(1 -\eps)pt \sqrt{\alpha \log_b d}$, as claimed. 
It remains to show that almost every graph $H$ with $t$ vertices and average degree $d$ is not a minor of $G$.

A \emph{blob} is a non-empty subset of $V(G_0)$. 
A \emph{blobbing} $(B_1,B_2,\ldots,B_t)$ is an ordered sequence of $t$ blobs with total size at most $|V(G)|$, such that each vertex of $G_0$ is in at most $r$ blobs. 

The motivation for these definitions is as follows: 
Suppose that a graph $H$ is a minor of $G$ and $V(H)=\{1,2,\dots,t\}$. Then for each vertex $v$ of $H$ there is a set $X_v\subseteq V(G)$, such that $X_v\cap X_w=\emptyset$ for distinct $v,w\in V(H)$, and for every edge $vw$ of $H$, there is an edge in $G$ between $X_v$ and $X_w$. For each vertex $v$ of $H$, let $B_v:= \{ x\in V(G_0): X_v \cap I_x\neq\emptyset\}$, called the \emph{projection} of $X_v$ to $G_0$. Note that $\sum_v |B_v| \leq \sum_v |X_v| \leq |V(G)|$, and each vertex of $G_0$ is in at most $r$ of $B_1,B_2,\dots,B_t$. Thus $(B_1,B_2,\dots,B_t)$ is a blobbing. Also note that by the construction of $G$, if $B_v\cap B_w=\emptyset$, then 
there is an edge of $G$ between $X_v$ and $X_w$ if and only if there is an edge of $G_0$ between $B_v$ and $B_w$. 

\begin{claim}\label{c:blob}
The number of blobbings is at most $(4d)^{t\ell}$. 
\end{claim}

\begin{proof}
For positive integers $d,t$ and for each positive integer $n\geq d$, let $g(d,t,n)$ be the number of $t$-tuples $(X_1,X_2,\dots,X_t)$ such that $X_i$ is a non-empty subset of $\{1,2,\dots,d\}$ for all $i\in\{1,2,\dots,t\}$, and $\sum_{i=1}^t|X_i| \leq n$. Below we prove that $g(d,t,n) \leq (4d)^n$ by induction on $t$. The result follows, since the number of blobbings is at most $g(d,t,|V(G)|) \leq g(d,t, \floor{t\ell} )$.  

In the base case, $g(d,1,n) \leq 2^d \leq (4d)^n$, as desired. Now assume the claim for $t-1$. 
Observe that 
\begin{equation*}
g(d,t,n) 
= \sum_{i=1}^d \binom{d}{i} g(d,t-1,n-i) .
\end{equation*}
By induction,
\begin{equation*}
g(d,t,n) 
\leq \sum_{i=1}^d \binom{d}{i} (4d)^{n-i}
\leq \sum_{i=1}^d \left(\frac{ed}{i} \right)^i (4d)^{n-i}
= (4d)^{n} \sum_{i=1}^d \left(\frac{e}{4i} \right)^i  
< (4d)^{n}.\end{equation*}
This completes the proof. 
\end{proof}

Two blobs are a \emph{good pair} if they are disjoint and non-adjacent $\ell$-sets in $G_0$. 

\begin{claim}\label{c:good}
Every blobbing has at least $\frac{\eps^2}{400} d^{-\alpha}t^2$ good pairs. 
\end{claim}

\begin{proof}
Suppose for a contradiction that some blobbing $(B_1,B_2,\dots,B_t)$ has less than 
$\frac{\eps^2}{400} d^{-\alpha}t^2$ good pairs.  
Let $X$ be the set of blobs $B_i$ such that $|B_i| \leq \ell$. 
Then $ \ell(t-|X| )  < |V(G)| < \left(1-\frac{\eps}{4}\right)\ell t$,  implying $|X| >\frac{\eps}{4} t$. 
Let $Y$ be the set of blobs in $X$ that belong to at most $\frac{\eps}{20} d^{-\alpha}t$ good pairs.  
Thus the total number of good pairs is at least $\frac{\eps}{40} d^{-\alpha} t |X \setminus Y|$, implying that $|X \setminus Y| < \frac{\eps}{10} t$ and $|Y| > \frac{3\eps}{20} t$.
   Let $Z$ be a maximal subset of $Y$ such that the blobs in $Z$ are pairwise
   disjoint and  contain at most $\frac12 d^{-\alpha}\binom{|Z|}{2}$ good pairs. 
   Then $1\leq |Z|< s$  by property $(\star)$ of $G_0$. 
%   \comment{Referee: It is not clear what the ‘above-mentioned’ property in line 7, page 5, so it would be better   	to state what part of Lemma 6 has been used. It may be even worth highlighting this part, since it seems to be the key part using Lemma 6.}   \comment{p5 l7: Which above-mentioned property?}
   Let $Z'$ be the set of blobs in $Y$ that are disjoint from every blob in $Z$. 
   Since each blob in $Z$ intersects at most $\ell r$ other blobs, 
   $|Y| \leq |Z'| + \ell r |Z| < |Z'| + \ell r s$, and $|Z'| >\frac{3\eps}{20} t - \ell r s \geq   \frac{\eps}{10} t$ for sufficiently large $d$. 
%   \comment{p5 l9: Give more details why $3\eps/20 t -\ell r s\geq \eps/10 t$.}
   By the maximality of $Z$,  every blob in $Z'$ is in a good pair with at least 
   $\frac12 d^{-\alpha}|Z|$ blobs in $Z$. 
 % \comment{p5 l10: Give more details why every blob in $Z'$ is in $xyz$ many good pairs with blobs in $Z$.}
   So in total there are at least
   $\frac12 d^{-\alpha}|Z||Z'|$ good pairs $\{B_i,B_j\}$ with $B_i \in Z$ and $B_j \in Z'$. 
   So some $B_i \in Z$ is in more than 
   $\frac12 d^{-\alpha}|Z'| > d^{-\alpha} \frac{\eps}{20} t$ good pairs, contradicting the definition of $Y$.
\end{proof}

Let $H$ be a graph with $V(H)=\{1,\ldots,t\}$. We say that a blobbing $(B_1,B_2,\ldots,B_t)$ is \emph{$H$-compatible} if for every $ij \in E(H)$ the blobs $B_i$ and $B_j$ intersect or are adjacent, implying that $\{B_i,B_j\}$ is not good. As explained above, if $H$ is a minor of $G$, then there exists an $H$-compatible blobbing. By \cref{c:good}, if $H\in\mathcal{G}(t,m)$ then the probability that a given blobbing is $H$-compatible is at most
\begin{equation*}
{\binom{\binom{t}{2}-\frac{\eps^2}{400}d^{-\alpha}t^2}{m}} / {\binom{\binom{t}{2}}{m}} 
\leq \left(1 -\frac{\eps^2d^{-\alpha}}{200}\right)^m 
\leq \exp\left(-\frac{\eps^2md^{-\alpha}}{200}\right).
\end{equation*}
Combining this inequality, \cref{c:blob} and the union bound, if $H\in\mathcal{G}(t,td/2)$ then 
%the  probability that $H$ is  a minor of $G$ is at most 
%OLD
%\begin{equation*}
%(4d)^{t\ell} \exp\left(-\frac{\eps^2 td^{1 -\alpha}}{400}\right)
%\end{equation*}
%which tends to $0$ as $d\rightarrow\infty$. 
%Hence almost every such graph $H$ is not a minor of $G$.
\begin{equation*}
\mathbb{P}(H\text{ is a minor of }G) \leq 
(4d)^{t\ell} \exp\left(-\frac{\eps^2 td^{1 -\alpha}}{400}\right),
\end{equation*}
which is less than $c^t$ for sufficiently large $d$. 
\end{proof}

\begin{proof}[Proof of \cref{main}.]
Choose $x > 0$ so that $\lambda = \frac{1-e^{-x}}{ \sqrt{x}}$, let $b:=e^x$, $p:=1-e^{-x}$. 
Let  $\alpha := (\frac{1-\eps}{1-\eps/2})^2$, implying $(1-\frac{\eps}{2})\sqrt{\alpha}=1-\eps$. By \cref{TheLemma},  there exists $d_0$ such that for every integer $d\geq d_0$ and for every integer $t\geq d+1$,  there is a graph $G$ with average degree at least $(1 -\frac{\eps}{2})p t \sqrt{\alpha \log_b d}$ such that 
%almost every graph $H$ with $t$ vertices and average degree $d$ is not a minor of $G$.	
if $H\in\mathcal{G}(t,td/2)$ then $\mathbb{P}(H\text{ is a minor of }G) < c^t$.
Since  
\begin{equation*}\left(1 -\frac{\eps}{2}\right)p t \sqrt{\alpha \log_b d} 
= \left(1 -\frac{\eps}{2}\right)\sqrt{\alpha} \left( \frac{1-e^{-x}}{\sqrt{x}} \right) t\sqrt{\ln d}
=(1-\eps)\lambda t\sqrt{\ln d},\end{equation*}
the graph $G$ satisfies the conditions of the theorem.
\end{proof}

%\cref{TheLemma} with $p=0.75133$ and $\alpha=\frac12 -\eps$, for some sufficiently small $\eps$, implies \cref{main}.

We finish with the natural open problem that arises from this work: Can the constant in the upper bound of \citet{ReedWood16} be improved to match the lower bound in the present paper? That is, is 
$f(H) \leq (\lambda + o(1) ) t\,\sqrt{\ln d}$ for every graph $H$ with $t$ vertices and average degree $d$? 

\subsection*{Note Added in Proof} 
Following the initial release of this paper, \citet{TW19} announced a solution to the above open problem. 

\subsection*{Acknowledgement} This research was partially completed at the Armenian Workshop  on Graphs, Combinatorics and Probability, June 2019. Many thanks to the other organisers and participants.

%\bibliographystyle{myNatbibStyle}
%\bibliography{myBibliography}

\begin{thebibliography}{33}
\providecommand{\natexlab}[1]{#1}
\providecommand{\url}[1]{\texttt{#1}}
\providecommand{\urlprefix}{}
\expandafter\ifx\csname urlstyle\endcsname\relax
  \providecommand{\doi}[1]{doi:\discretionary{}{}{}#1}\else
  \providecommand{\doi}{doi:\discretionary{}{}{}\begingroup
  \urlstyle{rm}\Url}\fi

\bibitem[{Alon and F\"{u}redi(1992)}]{AF92}
\textsc{Noga Alon and Zolt\'{a}n F\"{u}redi}.
\newblock Spanning subgraphs of random graphs.
\newblock \emph{Graphs Combin.}, 8(1):91--94, 1992.
\newblock \doi{10.1007/BF01271712}.
\newblock \msn{1157513}.

\bibitem[{Bollob{\'a}s et~al.(1980)Bollob{\'a}s, Catlin, and
  Erd{\H{o}}s}]{BCE80}
\textsc{B{\'e}la Bollob{\'a}s, Paul~A. Catlin, and Paul Erd{\H{o}}s}.
\newblock Hadwiger's conjecture is true for almost every graph.
\newblock \emph{European J. Combin.}, 1(3):195--199, 1980.
\newblock \doi{10.1016/S0195-6698(80)80001-1}.
\newblock \msn{593989}.

\bibitem[{Chudnovsky et~al.(2011)Chudnovsky, Reed, and Seymour}]{CRS11}
\textsc{Maria Chudnovsky, Bruce Reed, and Paul Seymour}.
\newblock The edge-density for {$K_{2,t}$} minors.
\newblock \emph{J. Combin. Theory Ser. B}, 101(1):18--46, 2011.
\newblock \doi{10.1016/j.jctb.2010.09.001}.
\newblock \msn{2737176}.

\bibitem[{Cs\'oka et~al.(2017)Cs\'oka, Lo, Norin, Wu, and Yepremyan}]{CLNWY17}
\textsc{Endre Cs\'oka, Irene Lo, Sergey Norin, Hehui Wu, and Liana Yepremyan}.
\newblock The extremal function for disconnected minors.
\newblock \emph{J. Combin. Theory Ser. B}, 126:162--174, 2017.
\newblock \doi{10.1016/j.jctb.2017.04.005}.
\newblock \msn{3667667}.

\bibitem[{de~la Vega(1983)}]{delaVega}
\textsc{W.~Fernandez de~la Vega}.
\newblock On the maximum density of graphs which have no subcontraction to
  {$K^{s}$}.
\newblock \emph{Discrete Math.}, 46(1):109--110, 1983.
\newblock \doi{10.1016/0012-365X(83)90280-7}.
\newblock \msn{0708172}.

\bibitem[{Dirac(1964)}]{Dirac64}
\textsc{Gabriel~A. Dirac}.
\newblock Homomorphism theorems for graphs.
\newblock \emph{Math. Ann.}, 153:69--80, 1964.
\newblock \doi{10.1007/BF01361708}.
\newblock \msn{0160203}.

\bibitem[{Fox(2011)}]{Fox11}
\textsc{Jacob Fox}.
\newblock Constructing dense graphs with sublinear {H}adwiger number.
\newblock 2011.
\newblock \arXiv{1108.4953}.

\bibitem[{Harvey and Wood(2015)}]{HarveyWood-Cycles}
\textsc{Daniel~J. Harvey and David~R. Wood}.
\newblock Cycles of given size in a dense graph.
\newblock \emph{SIAM J. Discrete Math.}, 29(4):2336--2349, 2015.
\newblock \doi{10.1137/15M100852X}.
\newblock \msn{3427042}.

\bibitem[{Harvey and Wood(2016)}]{HarveyWood16}
\textsc{Daniel~J. Harvey and David~R. Wood}.
\newblock Average degree conditions forcing a minor.
\newblock \emph{Electron. J. Combin.}, 23(1):\#P1.42, 2016.
\newblock \doi{10.37236/5321}.
\newblock \msn{3484747}.

\bibitem[{Hendrey et~al.(tion)Hendrey, Norin, and Wood}]{HNW}
\textsc{Kevin Hendrey, Sergey Norin, and David~R. Wood}.
\newblock On the extremal function for sparse minors.
\newblock In preparation.

\bibitem[{Hendrey and Wood(2018)}]{HW18a}
\textsc{Kevin Hendrey and David~R. Wood}.
\newblock The extremal function for {P}etersen minors.
\newblock \emph{J. Combin. Theory Ser. B}, 131:220--253, 2018.
\newblock \doi{10.1016/j.jctb.2018.02.001}.
\newblock \msn{3794147}.

\bibitem[{J{\o}rgensen(1994)}]{Jorgensen94}
\textsc{Leif~K. J{\o}rgensen}.
\newblock Contractions to {$K_8$}.
\newblock \emph{J. Graph Theory}, 18(5):431--448, 1994.
\newblock \doi{10.1002/jgt.3190180502}.
\newblock \msn{1283309}.

\bibitem[{Kapadia et~al.(2019)Kapadia, Norin, and Qian}]{KNQ19}
\textsc{Rohan Kapadia, Sergey Norin, and Yingjie Qian}.
\newblock Asymptotic density of graphs excluding disconnected minors.
\newblock 2019.
\newblock \arXiv{1903.03908}.

\bibitem[{Kostochka(1982)}]{Kostochka82}
\textsc{Alexandr~V. Kostochka}.
\newblock The minimum {H}adwiger number for graphs with a given mean degree of
  vertices.
\newblock \emph{Metody Diskret. Analiz.}, 38:37--58, 1982.
\newblock \msn{0713722}.

\bibitem[{Kostochka(1984)}]{Kostochka84}
\textsc{Alexandr~V. Kostochka}.
\newblock Lower bound of the {H}adwiger number of graphs by their average
  degree.
\newblock \emph{Combinatorica}, 4(4):307--316, 1984.
\newblock \doi{10.1007/BF02579141}.
\newblock \msn{0779891}.

\bibitem[{Kostochka and Prince(2008)}]{KP08}
\textsc{Alexandr~V. Kostochka and Noah Prince}.
\newblock On {$K_{s,t}$}-minors in graphs with given average degree.
\newblock \emph{Discrete Math.}, 308(19):4435--4445, 2008.
\newblock \doi{10.1016/j.disc.2007.08.041}.
\newblock \msn{2433771}.

\bibitem[{Kostochka and Prince(2010)}]{KP10}
\textsc{Alexandr~V. Kostochka and Noah Prince}.
\newblock Dense graphs have {$K_{3,t}$} minors.
\newblock \emph{Discrete Math.}, 310(20):2637--2654, 2010.
\newblock \doi{10.1016/j.disc.2010.03.026}.
\newblock \msn{2672210}.

\bibitem[{Kostochka and Prince(2012)}]{KP12}
\textsc{Alexandr~V. Kostochka and Noah Prince}.
\newblock On {$K_{s,t}$}-minors in graphs with given average degree, {II}.
\newblock \emph{Discrete Math.}, 312(24):3517--3522, 2012.
\newblock \doi{10.1016/j.disc.2012.08.004}.
\newblock \msn{2979480}.

\bibitem[{K{\"u}hn and Osthus(2005)}]{KO05}
\textsc{Daniela K{\"u}hn and Deryk Osthus}.
\newblock Forcing unbalanced complete bipartite minors.
\newblock \emph{European J. Combin.}, 26(1):75--81, 2005.
\newblock \doi{10.1016/j.ejc.2004.02.002}.
\newblock \msn{2101035}.

\bibitem[{Mader(1967)}]{Mader67}
\textsc{Wolfang Mader}.
\newblock Homomorphieeigenschaften und mittlere {K}antendichte von {G}raphen.
\newblock \emph{Math. Ann.}, 174:265--268, 1967.
\newblock \doi{10.1007/BF01364272}.
\newblock \msn{0220616}.

\bibitem[{Mader(1968)}]{Mader68}
\textsc{Wolfgang Mader}.
\newblock Homomorphies\"atze f\"ur {G}raphen.
\newblock \emph{Math. Ann.}, 178:154--168, 1968.
\newblock \doi{10.1007/BF01350657}.
\newblock \msn{0229550}.

\bibitem[{Mitzenmacher and Upfal(2005)}]{MU05}
\textsc{Michael Mitzenmacher and Eli Upfal}.
\newblock \emph{Probability and computing}.
\newblock Cambridge University Press, 2005.
\newblock \doi{10.1017/CBO9780511813603}.

\bibitem[{Myers(2002)}]{Myers-CPC02}
\textsc{Joseph~Samuel Myers}.
\newblock Graphs without large complete minors are quasi-random.
\newblock \emph{Combin. Probab. Comput.}, 11(6):571--585, 2002.
\newblock \doi{10.1017/S096354830200531X}.
\newblock \msn{1940121}.

\bibitem[{Myers(2003)}]{Myers03}
\textsc{Joseph~Samuel Myers}.
\newblock The extremal function for unbalanced bipartite minors.
\newblock \emph{Discrete Math.}, 271(1-3):209--222, 2003.
\newblock \doi{10.1016/S0012-365X(03)00051-7}.
\newblock \msn{1999544}.

\bibitem[{Myers and Thomason(2005)}]{MT-Comb05}
\textsc{Joseph~Samuel Myers and Andrew Thomason}.
\newblock The extremal function for noncomplete minors.
\newblock \emph{Combinatorica}, 25(6):725--753, 2005.
\newblock \doi{10.1007/s00493-005-0044-0}.
\newblock \msn{2199433}.

\bibitem[{Reed and Wood(2016{\natexlab{a}})}]{ReedWood16}
\textsc{Bruce Reed and David~R. Wood}.
\newblock Forcing a sparse minor.
\newblock \emph{Combin. Probab. Comput.}, 25(2):300--322, 2016{\natexlab{a}}.
\newblock \doi{10.1017/S0963548315000073}.
\newblock \msn{3455678}.

\bibitem[{Reed and Wood(2016{\natexlab{b}})}]{ReedWood16c}
\textsc{Bruce Reed and David~R. Wood}.
\newblock `{F}orcing a sparse minor'---{C}orrigendum.
\newblock \emph{Combin. Probab. Comput.}, 25(2):323, 2016{\natexlab{b}}.
\newblock \doi{10.1017/S0963548315000115}.
\newblock \msn{3455679}.

\bibitem[{Song and Thomas(2006)}]{ST06}
\textsc{Zi-Xia Song and Robin Thomas}.
\newblock The extremal function for {$K_9$} minors.
\newblock \emph{J. Combin. Theory Ser. B}, 96(2):240--252, 2006.
\newblock \doi{10.1016/j.jctb.2005.07.008}.
\newblock \msn{2208353}.

\bibitem[{Song(2004)}]{Song04}
\textsc{Zixia Song}.
\newblock \emph{Extremal Functions for Contractions of Graphs}.
\newblock Ph.D. thesis, Georgia Institute of Technology, USA, 2004.
\newblock \msn{2706190}.

\bibitem[{Thomason(1984)}]{Thomason84}
\textsc{Andrew Thomason}.
\newblock An extremal function for contractions of graphs.
\newblock \emph{Math. Proc. Cambridge Philos. Soc.}, 95(2):261--265, 1984.
\newblock \doi{10.1017/S0305004100061521}.
\newblock \msn{0735367}.

\bibitem[{Thomason(2001)}]{Thomason01}
\textsc{Andrew Thomason}.
\newblock The extremal function for complete minors.
\newblock \emph{J. Combin. Theory Ser. B}, 81(2):318--338, 2001.
\newblock \doi{10.1006/jctb.2000.2013}.
\newblock \msn{1814910}.

\bibitem[{Thomason(2006)}]{Thomason06}
\textsc{Andrew Thomason}.
\newblock Extremal functions for graph minors.
\newblock In \emph{More sets, graphs and numbers}, vol.~15 of \emph{Bolyai Soc.
  Math. Stud.}, pp. 359--380. Springer, Berlin, 2006.

\bibitem[{Thomason(2008)}]{Thom08}
\textsc{Andrew Thomason}.
\newblock Disjoint unions of complete minors.
\newblock \emph{Discrete Math.}, 308(19):4370--4377, 2008.
\newblock \doi{10.1016/j.disc.2007.08.021}.
\newblock \msn{2433863}.

\bibitem[{Thomason and Wales(2019)}]{TW19}
\textsc{Andrew Thomason} and \textsc{Matthew Wales}.
\newblock On the extremal function for graph minors
\newblock \arXiv{1907.11626}, 2019. 

\end{thebibliography}

\def\soft#1{\leavevmode\setbox0=\hbox{h}\dimen7=\ht0\advance \dimen7
  by-1ex\relax\if t#1\relax\rlap{\raise.6\dimen7
  \hbox{\kern.3ex\char'47}}#1\relax\else\if T#1\relax
  \rlap{\raise.5\dimen7\hbox{\kern1.3ex\char'47}}#1\relax \else\if
  d#1\relax\rlap{\raise.5\dimen7\hbox{\kern.9ex \char'47}}#1\relax\else\if
  D#1\relax\rlap{\raise.5\dimen7 \hbox{\kern1.4ex\char'47}}#1\relax\else\if
  l#1\relax \rlap{\raise.5\dimen7\hbox{\kern.4ex\char'47}}#1\relax \else\if
  L#1\relax\rlap{\raise.5\dimen7\hbox{\kern.7ex
  \char'47}}#1\relax\else\message{accent \string\soft \space #1 not
  defined!}#1\relax\fi\fi\fi\fi\fi\fi}

\end{document}